\documentclass[11pt,reqno]{amsart}

\usepackage{amsmath,amsthm,amsfonts,amssymb}
\usepackage[colorlinks]{hyperref}
\usepackage[T1]{fontenc}
\usepackage[mathscr]{eucal}

\newtheorem{theorem}{Theorem}[section]
\newtheorem{corollary}[theorem]{Corollary}
\newtheorem{definition}[theorem]{Definition}
\newtheorem{example}[theorem]{Example}
\newtheorem{lemma}[theorem]{Lemma}
\newtheorem{proposition}[theorem]{Proposition}
\newtheorem{remark}[theorem]{Remark}

\numberwithin{equation}{section}

\begin{document}

\title[Hardy-Rogers and Jungck type theorems]
{Hardy-Rogers and Jungck Type Fixed Point Theorems in Perturbed Metric Spaces, Stability and Data Dependence}

\date{}

\author[Du\v{s}an Bajovi\'c]{Du\v{s}an Bajovi\'c}
\author[Zoran Mitrovi\'c]{Zoran Mitrovi\'c}
\author[Boris Petkovi\'c]{Boris Petkovi\'c}

\begin{abstract} 
In this paper we establish Hardy-Rogers and Jungck type fixed point theorems in perturbed metric spaces, where the observed distance $D$ is separated from the exact metric $d$ by a nonnegative perturbation $P$. Rather than the uniform absorption of $P$ by $d$ required under domination, we examine a weaker demand, imposed only on the pairs of points that appear with a contractive coefficient. We show that without some such condition, and without a continuity hypothesis on $T$, a perturbed Banach contraction on a complete perturbed metric space may fail to have a fixed point. We further prove Ulam-Hyers stability, well-posedness, and data dependence results in which residuals are measured in the observed distance $D$, with all constants explicit, and we derive a priori error estimates for the Picard and Jungck iterations computable from observed data. The Jungck type theorem is established for weakly compatible pairs.
\end{abstract}

\maketitle

\textbf{Mathematics Subject Classification (2020):}
47H10, 54H25, 54E35.

\textbf{Key words and phrases:}
Perturbed metric space; Hardy-Rogers contraction; Jungck fixed point;
common fixed point; Ulam-Hyers stability; well-posedness; data dependence.

\section{Introduction}

Fixed point theory is one of the central tools of nonlinear analysis, with
applications to differential and integral equations, optimization, numerical
methods and dynamical systems. The Banach contraction principle
\cite{Banach1922} and its many extensions, for instance those due to Reich and
Hardy-Rogers \cite{Reich1972,HardyRogers1973}, give robust criteria under
which nonlinear mappings admit unique fixed points and simple iterative
schemes converge to them.

A recent direction in the metric approach to fixed points is the study of
perturbed metric spaces, introduced by Jleli and Samet \cite{JleliSamet2025}. In this setting the observed distance between two points is modeled as
\begin{equation}\label{eq:intro-D} D(x,y)=d(x,y)+P(x,y),
\end{equation}
where $d$ is a genuine metric and $P\ge 0$ is a perturbation term
representing, for example, measurement noise, discretization effects or
modeling error. The essential point is that convergence and completeness are
always understood with respect to the exact metric $d=D-P$, whereas the
contractive assumptions may be written in terms of the observed distance $D$.
The framework has attracted immediate attention. Nu\c{t}u and P\u{a}curar
\cite{NutuPacurar2025} studied $\varphi$-perturbed and Kannan-perturbed
mappings, Chatterjea-type conditions in a multiplicative variant of the
framework are considered in \cite{ArxivNewType}, perturbed $b$-metric and
perturbed extended $b$-metric spaces appear in \cite{ArxivPerturbedB}, a $C^{*}$-algebra-valued version is developed in \cite{Tijjani2025} and an
application to an integral equation is given by Dobri\c{t}oiu \cite{Dobritoiu2024}. See also the survey \cite{MoussaouiRadenovic2025} for generalized distances. \\

In general, $D$ does not need to be a metric. It may fail symmetry, the triangle
inequality, and even the vanishing of the diagonal, since
$D(x,x)=P(x,x)$ can be positive. Because of this, one cannot directly repeat
metric space proofs in terms of $D$ without additional assumptions. In
\cite{JleliSamet2025} the Banach-type theorem is proved under the assumption
that $T$ is continuous with respect to the exact metric. In this paper we
work instead under a coefficient weighted absorption condition, condition 
$(\mathrm{W})$ below (Definition~\ref{def:W}), whose simplest
instance is a domination condition of the form $D\le M\,d$. Both
conditions still permit nonmetric, nonsymmetric observed distances but
allow fixed point arguments to be carried out rigorously in the exact
metric space $(X,d)$, with no continuity assumption on $T$. Condition
$(\mathrm{W})$ has the additional feature that it constrains the perturbation
$P$ only on the pairs of points actually weighted by the contractive
coefficients, and is strictly weaker than domination.

Our contributions are the following.

\begin{enumerate}
	\item[(i)] We prove a Hardy-Rogers type theorem
	(Theorem~\ref{thm:HR}) under condition $(\mathrm{W})$ and the
	coefficient condition $(a+a_0)+(b+b_0)+(c+c_0)+(h+h_0)+(e+e_0)<1$;
	in the dominated case $D\le M\,d$ this reduces to $M(a+b+c+h+e)<1$
	(Corollary~\ref{cor:dominated}), which reproduces the full classical
	Hardy-Rogers condition \cite{HardyRogers1973} with the domination
	constant $M$ entering linearly. Condition $(\mathrm{W})$ is strictly
	weaker than domination (Remark~\ref{rem:W-weaker},
	Example~\ref{ex:strict}). The corresponding Reich, Kannan, Chatterjea
	and Banach type results follow as corollaries. An a priori error
	estimate for the Picard iteration, computable from the
	observed initial displacement $D(x_0,x_1)$, is derived
	(Proposition~\ref{prop:error}).
	\item[(ii)] We show by an explicit counterexample
	(Example~\ref{ex:counter}) that if both condition $(\mathrm{W})$ (in
	particular, the domination condition) and
	the continuity of $T$ are dropped, then a perturbed Banach contraction
	on a complete perturbed metric space may have no fixed point.
	Consequently, some hypothesis of this kind is indispensable. in
	particular, Theorem~2.2 of \cite{NutuPacurar2025}, which is stated
	without any such hypothesis, requires an additional assumption (see below
	Remark~\ref{rem:NP}).
	\item[(iii)] We initiate the quantitative stability theory of the fixed
	point problem in perturbed metric spaces. Ulam-Hyers stability with
	residuals measured in the observed distance $D$
	(Theorem~\ref{thm:UH}), well-posedness
	(Corollary~\ref{cor:wp}), and data dependence with the discrepancy of
	the operators measured in $D$ (Theorem~\ref{thm:dd}). All constants are explicit.
	\item[(iv)] We prove a Jungck type common fixed point theorem for weakly compatible pairs (Theorem~\ref{thm:Jungck}), which is
	strictly weaker than commutativity, together with an observed-data
	error estimate for the Jungck iteration
	(Proposition~\ref{prop:Jungck-error}).
\end{enumerate}

To the best of our knowledge, Hardy-Rogers and Jungck type theorems, as well
as Ulam-Hyers stability and data dependence, have not previously been
considered in perturbed metric spaces.

\section{Preliminaries}

We now recall the basic terminology used throughout.

\begin{definition}[Perturbed metric space \cite{JleliSamet2025}]\label{def:PMS}
	Let $D,P:X\times X\to [0,\infty)$.
	We say that $(X,D,P)$ is a perturbed metric space if
	\begin{equation}\label{eq:def-d}
		d(x,y):=D(x,y)-P(x,y)
	\end{equation}
	is a metric on $X$.
	The metric $d$ is called the exact metric.
\end{definition}

\begin{definition}\label{def:conv}
	Let $(X,D,P)$ be a perturbed metric space with exact metric $d=D-P$.
	\begin{itemize}
		\item[$(i)$] A sequence $(x_n)$ converges to $x$ if $d(x_n,x)\to 0$.
		\item[$(ii)$] $(x_n)$ is Cauchy if it is Cauchy in $(X,d)$.
		\item[$(iii)$] $(X,D,P)$ is complete if $(X,d)$ is complete.
	\end{itemize}
\end{definition}

\begin{definition}[Controlled observed distance]\label{def:controlled}
	Let $(X,D,P)$ be a perturbed metric space with exact metric $d=D-P$.
	We say that the observed distance $D$ is controlled by $d$ if
	there exists a constant $M\ge 1$ such that
	\begin{equation}\label{eq:controlled}
		D(x,y)\le M\,d(x,y),\qquad x,y\in X.
	\end{equation}
\end{definition}

Since $P\ge 0$, we always have $d\le D$. Hence, under
Definition~\ref{def:controlled},
\begin{equation}\label{eq:equiv}
	d(x,y)\le D(x,y)\le M\,d(x,y),\qquad x,y\in X.
\end{equation}
The following elementary observations clarify the exact scope of the
domination condition.

\begin{proposition}\label{prop:structure}
	Let $(X,D,P)$ be a perturbed metric space with exact metric $d$. The
	following statements are equivalent:
	\begin{itemize}
		\item[$(i)$] $D$ is controlled by $d$ with constant $M$;
		\item[$(ii)$] $P(x,y)\le (M-1)\,d(x,y)$ for all $x,y\in X$.
	\end{itemize}
	Moreover, if these hold, then:
	\begin{itemize}
		\item[$(a)$] $D(x,x)=0$ for all $x\in X$ (so the perturbation
		necessarily vanishes on the diagonal);
		\item[$(b)$] $D$ satisfies the relaxed triangle inequality
		$D(x,z)\le M\bigl(D(x,y)+D(y,z)\bigr)$ for all $x,y,z\in X$;
		\item[$(c)$] $D$ and $d$ determine the same convergent sequences
		and the same Cauchy sequences.
	\end{itemize}
\end{proposition}

\begin{proof}
	The equivalence of $(i)$ and $(ii)$ is immediate from $D=d+P$. For
	$(a)$, take $x=y$ in \eqref{eq:controlled} to get
	$D(x,x)\le M\,d(x,x)=0$. For $(b)$, using \eqref{eq:equiv} and the
	triangle inequality of $d$,
	\[
	D(x,z)\le M\,d(x,z)\le M\bigl(d(x,y)+d(y,z)\bigr)
	\le M\bigl(D(x,y)+D(y,z)\bigr).
	\]
	Claim $(c)$ follows directly from \eqref{eq:equiv}.
\end{proof}

\begin{remark}\label{rem:interpretation}
	Statement $(ii)$ of Proposition~\ref{prop:structure} gives the natural
	interpretation of the domination condition. Tthe perturbation is
	relatively bounded by the exact distance, with relative bound
	$M-1$. This is precisely the standard model of relative measurement
	error. An instrument (or a floating-point computation) that reports distances with relative error at most $\varepsilon$ produces an observed distance $D$ controlled by $d$ with $M=1+\varepsilon$. Note
	also that, by $(b)$, a controlled observed distance always satisfies a
	relaxed triangle inequality with coefficient $M$. When $P$ is
	symmetric (so $D$ is symmetric), $(a)$ and $(b)$ show that $D$ is then
	a genuine $b$-metric with coefficient $M$ in the sense of
	\cite{Czerwik1993}. When $P$ is not symmetric  as in
	Example~\ref{ex:HR-nonsym} below - $D$ satisfies (b3) but not (b2),
	so it is only quasi-$b$-metric with coefficient $M$. Tthe
	asymmetry of $P$ is exactly the source of the ``quasi''. In either
	case the framework is more special than a bare (quasi-)$b$-metric,
	since equivalence \eqref{eq:equiv} with a genuine metric $d$ is built
	in from the start.
\end{remark}

The domination condition constrains the perturbation $P$ uniformly on
all pairs of points. For the fixed point theorems below, this is more
than is needed. The contractive inequality weights only certain pairs, and it
suffices to absorb $P$ into $d$ on those pairs. This motivates the following
condition, which will be our standing hypothesis.

\begin{definition}[Condition $(\mathrm{W})$]\label{def:W}
	Let $(X,D,P)$ be a perturbed metric space with exact metric $d=D-P$,
	let $T:X\to X$ and let $a,b,c,h,e\ge 0$. We say that $T$ satisfies
	condition $(\mathrm{W})$ (with respect to the coefficients
	$a,b,c,h,e$) if there exist constants $a_0,b_0,c_0,h_0,e_0\ge0$ such
	that, for all $x,y\in X$,
	\begin{multline}\label{eq:W}
		a\,P(x,y)+b\,P(x,Tx)+c\,P(y,Ty)+h\,P(x,Ty)+e\,P(y,Tx)\\
		\le\ a_0\,d(x,y)+b_0\,d(x,Tx)+c_0\,d(y,Ty)+h_0\,d(x,Ty)+e_0\,d(y,Tx).
	\end{multline}
\end{definition}

Note that \eqref{eq:W} involves the perturbation $P$ only on the pairs that
carry a nonzero coefficient in the contractive condition. In particular, if
some of $a,b,c,h,e$ vanish, the corresponding pairs are entirely
unconstrained. No symmetry of $P$ (or of $D$) is assumed.

We shall repeatedly use the following conversion lemma.

\begin{lemma}\label{lem:conversion}
	Let $(X,D,P)$ be a perturbed metric space with exact metric $d$. Let
	$T:X\to X$ satisfy, for all $x,y\in X$,
	\begin{equation}\label{eq:HR-D}
		D(Tx,Ty)\le a\,D(x,y)+b\,D(x,Tx)+c\,D(y,Ty)+h\,D(x,Ty)+e\,D(y,Tx),
	\end{equation}
	where $a,b,c,h,e\ge 0$, together with condition $(\mathrm{W})$ with
	constants $a_0,b_0,c_0,$ $h_0,e_0\ge0$. Set
	\[
	A=a+a_0,\quad B=b+b_0,\quad C=c+c_0,\quad H=h+h_0,\quad E=e+e_0.
	\]
	Then, for all $x,y\in X$,
	\begin{equation}\label{eq:HR-d}
		d(Tx,Ty)\le A\,d(x,y)+B\,d(x,Tx)+C\,d(y,Ty)+H\,d(x,Ty)+E\,d(y,Tx),
	\end{equation}
	and also, applying \eqref{eq:HR-D} and \eqref{eq:W} to the ordered
	pair $(y,x)$ and using the symmetry of $d$,
	\begin{equation}\label{eq:HR-d-swap}
		d(Tx,Ty)\le A\,d(x,y)+B\,d(y,Ty)+C\,d(x,Tx)+H\,d(y,Tx)+E\,d(x,Ty).
	\end{equation}
\end{lemma}

\begin{proof}
	Since $P\ge0$ gives $d\le D$, and $D=d+P$ termwise on the right of
	\eqref{eq:HR-D},
	\begin{align*}
	d(Tx,Ty)&\le D(Tx,Ty)\\
	&\le a\,d(x,y)+b\,d(x,Tx)+c\,d(y,Ty)+h\,d(x,Ty)+e\,d(y,Tx)\\
	&\quad+a\,P(x,y)+b\,P(x,Tx)+c\,P(y,Ty)+h\,P(x,Ty)+e\,P(y,Tx),
	\end{align*}
	and bounding the second line by \eqref{eq:W} gives \eqref{eq:HR-d}.
	Applying the same argument to the ordered pair $(y,x)$ - both
	\eqref{eq:HR-D} and \eqref{eq:W} hold for all ordered pairs - and
	using $d(Ty,Tx)=d(Tx,Ty)$ together with $d(y,x)=d(x,y)$ gives
	\eqref{eq:HR-d-swap}. No symmetry of $D$ or of $P$ is used.
\end{proof}

\begin{remark}[Domination implies condition $(\mathrm{W})$]\label{rem:dom-implies-W}
	If $D$ is controlled by $d$ with constant $M\ge1$, then
	$P\le(M-1)d$ by Proposition~\ref{prop:structure}$(ii)$, and
	\eqref{eq:W} holds with
	\[
	a_0=(M-1)a,\quad b_0=(M-1)b,\quad c_0=(M-1)c,
	\]
	\[
	h_0=(M-1)h,\quad e_0=(M-1)e,
	\]
	so that $A=Ma$, $B=Mb$, $C=Mc$, $H=Mh$, $E=Me$ in
	Lemma~\ref{lem:conversion}. Thus the lemma contains the
	domination-based conversion as a special case. The converse fails. The condition $(\mathrm{W})$ is strictly weaker than domination (see
	Remark~\ref{rem:W-weaker} and Example~\ref{ex:strict} below).
\end{remark}

\section{Main results}

\subsection{A Hardy-Rogers type theorem}

\begin{theorem}[Hardy-Rogers type theorem]\label{thm:HR}
	Let $(X,D,P)$ be a complete perturbed metric space with exact metric
	$d=D-P$. Let $T:X\to X$ satisfy \eqref{eq:HR-D} with
	$a,b,c,h,e\ge 0$, together with condition $(\mathrm{W})$ with
	constants $a_0,b_0,c_0,h_0,e_0\ge0$, and suppose
	\begin{equation}\label{eq:HR-coeff}
		(a+a_0)+(b+b_0)+(c+c_0)+(h+h_0)+(e+e_0)<1 .
	\end{equation}
	Then $T$ has a unique fixed point $x^\ast\in X$, and for every
	$x_0\in X$ the Picard sequence $x_{n+1}=Tx_n$ converges to $x^\ast$
	in $(X,d)$, with
	\begin{equation}\label{eq:kappa}
		d(x_{n},x_{n+1})\le\kappa^{\,n}\,d(x_0,x_1),\qquad
		\kappa:=\frac{2A+B+C+H+E}{\,2-(B+C+H+E)\,}\in[0,1),
	\end{equation}
	where $A,B,C,H,E$ are as in Lemma~\ref{lem:conversion}.
\end{theorem}

\begin{proof}
	Write $\sigma:=A+B+C+H+E<1$.
	
	\emph{Step 1: the Picard sequence is Cauchy.}
	Fix $x_0\in X$, let $x_{n+1}=Tx_n$ and put
	$\delta_n:=d(x_n,x_{n+1})$. Apply \eqref{eq:HR-d} and
	\eqref{eq:HR-d-swap} with $x=x_{n-1}$, $y=x_n$ and add the two
	inequalities. Since $d(x_n,Tx_{n-1})=d(x_n,x_n)=0$ and, by the
	triangle inequality for $d$,
	$d(x_{n-1},Tx_n)=d(x_{n-1},x_{n+1})\le\delta_{n-1}+\delta_n$, we
	obtain
	\[
	2\delta_n\le 2A\,\delta_{n-1}+(B+C)\,(\delta_{n-1}+\delta_n)
	+(H+E)\,(\delta_{n-1}+\delta_n),
	\]
	that is,
	\[
	\bigl(2-(B+C+H+E)\bigr)\,\delta_n
	\le\bigl(2A+B+C+H+E\bigr)\,\delta_{n-1}.
	\]
	Hence $\delta_n\le\kappa\,\delta_{n-1}$ with $\kappa$ as in
	\eqref{eq:kappa}, and $\kappa<1$ since $\sigma<1$. 	Thus $\delta_n\le\kappa^{\,n}\delta_0$, and for $m>n$,
	\[
	d(x_n,x_m)\le\sum_{k=n}^{m-1}\delta_k
	\le\delta_0\,\frac{\kappa^{\,n}}{1-\kappa}\ \xrightarrow[n\to\infty]{}0 .
	\]
	So $(x_n)$ is Cauchy in the complete metric space $(X,d)$. Let
	$x_n\to x^\ast$.
	
	\emph{Step 2: $x^\ast$ is a fixed point.} Applying \eqref{eq:HR-d}
	with $x=x_n$, $y=x^\ast$,
	\begin{align*}
	d(x_{n+1},Tx^\ast)
	&\le A\,d(x_n,x^\ast)+B\,\delta_n+C\,d(x^\ast,Tx^\ast)\\
	&\qquad +H\,d(x_n,Tx^\ast)+E\,d(x^\ast,x_{n+1}).
	\end{align*}
	Using $d(x_n,Tx^\ast)\le d(x_n,x^\ast)+d(x^\ast,Tx^\ast)$, this gives
	\[
	d(x_{n+1},Tx^\ast)
	\le (A+H)\,d(x_n,x^\ast)+B\,\delta_n+E\,d(x^\ast,x_{n+1})
	+(C+H)\,d(x^\ast,Tx^\ast).
	\]
	Since $d(x^\ast,Tx^\ast)\le d(x^\ast,x_{n+1})+d(x_{n+1},Tx^\ast)$,
	letting $n\to\infty$ and using $x_n\to x^\ast$, $\delta_n\to0$ yields
	\[
	d(x^\ast,Tx^\ast)\le (C+H)\,d(x^\ast,Tx^\ast),
	\]
	and $C+H<1$ forces $d(x^\ast,Tx^\ast)=0$, i.e., $Tx^\ast=x^\ast$.
	Note that no continuity of $T$ was used.
	
	\emph{Step 3: uniqueness.} If $u,v$ are fixed points, then
	\eqref{eq:HR-d} gives
	\[
	d(u,v)=d(Tu,Tv)\le A\,d(u,v)+H\,d(u,v)+E\,d(u,v)
	=(A+H+E)\,d(u,v),
	\]
	and $A+H+E<1$ forces $d(u,v)=0$, i.e., $u=v$.
\end{proof}

\begin{corollary}[Dominated case]\label{cor:dominated}
	Let $(X,D,P)$ be a complete perturbed metric space whose observed
	distance $D$ is controlled by $d$ with constant $M\ge1$, and let
	$T:X\to X$ satisfy \eqref{eq:HR-D} with $a,b,c,h,e\ge 0$ and
	\begin{equation}\label{eq:HR-coeff-M}
		M\,(a+b+c+h+e)<1 .
	\end{equation}
	Then all the conclusions of Theorem~\ref{thm:HR} hold, with
	$A=Ma$, $B=Mb$, $C=Mc$, $H=Mh$, $E=Me$.
\end{corollary}

\begin{proof}
	By Remark~\ref{rem:dom-implies-W}, condition $(\mathrm{W})$ holds
	with $a_0=(M-1)a,\dots,e_0=(M-1)e$, and then the left-hand side of
	\eqref{eq:HR-coeff} equals $M(a+b+c+h+e)<1$. Apply
	Theorem~\ref{thm:HR}.
\end{proof}

\begin{remark}[Condition $(\mathrm{W})$ is strictly weaker than
domination]\label{rem:W-weaker}
	The domination condition \eqref{eq:controlled} is thus the special
	case of condition $(\mathrm{W})$ in which the same uniform relative
	bound $P\le(M-1)d$ is imposed on every pair of points and
	distributed proportionally over the five slots of \eqref{eq:W}. In
	general, condition $(\mathrm{W})$ is strictly weaker, for two
	reasons. First, it constrains $P$ only on the pairs weighted by
	nonzero coefficients: for instance, in the Kannan case
	$a=h=e=0$ the inequality \eqref{eq:W} involves $P$ only on the graph
	pairs $(x,Tx)$, so the rbital bound
	$P(x,Tx)\le K\,d(x,Tx)$ for all $x\in X$ suffices (with
	$b_0=Kb$, $c_0=Kc$), and no condition whatsoever is imposed on
	$P(x,y)$ for general pairs. Second, it allows different relative
	bounds in different slots, whereas domination forces the same
	constant $M$ throughout. Example~\ref{ex:strict} exhibits a mapping
	to which Theorem~\ref{thm:HR} applies although $D$ is not controlled
	by $d$ for any constant $M$, so Corollary~\ref{cor:dominated} (and
	with it the entire domination framework) is not applicable.
\end{remark}

\begin{example}[Strictness]\label{ex:strict}
	Let $X=[0,1]$ with $d(x,y)=|x-y|$, so that $(X,d)$ is complete.
	Define
	\[
	P(x,y):=
	\begin{cases}
		1, & \text{if } x,y\in(\tfrac14,1] \text{ and } x\ne y,\\
		0, & \text{otherwise},
	\end{cases}
	\qquad D:=d+P ,
	\]
	and $T(x):=x/4$. Since $P(x,y)/d(x,y)=1/|x-y|$ is unbounded for
	$x,y\in(\tfrac14,1]$ close to each other, $D$ is not
	controlled by $d$ for any $M\ge1$, and Corollary~\ref{cor:dominated}
	does not apply. On the other hand, $Tx=x/4\le\tfrac14$ for every
	$x\in X$, so $P$ vanishes on all pairs $(Tx,Ty)$ and on all graph
	pairs $(x,Tx)$. Consequently
	\begin{align*}
	D(Tx,Ty)&=\tfrac14|x-y|\le\tfrac14(x+y)
	=\tfrac13\bigl(\tfrac34 x+\tfrac34 y\bigr)\\
	&=\tfrac13\bigl(d(x,Tx)+d(y,Ty)\bigr)
	\le\tfrac13\bigl(D(x,Tx)+D(y,Ty)\bigr),
	\end{align*}
	i.e., $T$ satisfies the Kannan type condition \eqref{eq:HR-D} with
	$b=c=\tfrac13$, $a=h=e=0$. Condition $(\mathrm{W})$ holds trivially
	with $a_0=b_0=c_0=h_0=e_0=0$, because its left-hand side
	$\tfrac13 P(x,Tx)+\tfrac13 P(y,Ty)$ vanishes identically. The
	coefficient condition \eqref{eq:HR-coeff} reads
	$\tfrac13+\tfrac13=\tfrac23<1$, so Theorem~\ref{thm:HR} applies and
	yields the unique fixed point $x^\ast=0$ with Picard convergence
	from every starting point.
\end{example}

\begin{remark}
	In the dominated case, the coefficient condition
	\eqref{eq:HR-coeff-M} matches the classical
	Hardy-Rogers condition $a+b+c+h+e<1$ of \cite{HardyRogers1973}, with
	the domination constant $M$ entering linearly. The symmetrization
	device in Step~1 of the proof of Theorem~\ref{thm:HR}, which avoids
	any doubling of the coefficient $h$,
	is classical in the metric setting (see, e.g., the discussion in
	\cite{HardyRogers1973}). The point here is that it survives the
	passage through a non-symmetric observed distance, because
	Lemma~\ref{lem:conversion} produces the two ordered $d$-inequalities
	\eqref{eq:HR-d} and \eqref{eq:HR-d-swap} regardless of any symmetry
	of $D$ or of $P$.
\end{remark}

\begin{corollary}[Reich type]\label{cor:Reich}
	Under the hypotheses of Corollary~\ref{cor:dominated} with $h=e=0$: if
	\[
	D(Tx,Ty)\le a\,D(x,y)+b\,D(x,Tx)+c\,D(y,Ty),\qquad x,y\in X,
	\]
	and $M(a+b+c)<1$, then $T$ has a unique fixed point and the Picard
	iteration converges to it in $(X,d)$.
\end{corollary}

\begin{corollary}[Kannan type]\label{cor:Kannan}
	Under the hypotheses of Corollary~\ref{cor:dominated} with $a=h=e=0$: if
	\[
	D(Tx,Ty)\le b\,\left(D(x,Tx)+D(y,Ty)\right),\qquad x,y\in X,
	\]
	and $2Mb<1$, then $T$ has a unique fixed point and the Picard
	iteration converges to it in $(X,d)$.
\end{corollary}

\begin{corollary}[Chatterjea type]\label{cor:Chatterjea}
	Under the hypotheses of Corollary~\ref{cor:dominated} with $a=b=c=0$: if
	\[
	D(Tx,Ty)\le h\,\left(D(x,Ty)+D(y,Tx)\right),\qquad x,y\in X,
	\]
	and $2Mh<1$, then $T$ has a unique fixed point and the Picard
	iteration converges to it in $(X,d)$.
\end{corollary}

\begin{corollary}[Banach type]\label{cor:Banach}
	Under the hypotheses of Corollary~\ref{cor:dominated} with $b=c=h=e=0$: if
	\[
	D(Tx,Ty)\le \lambda\, D(x,y),\qquad x,y\in X,
	\]
	for some $\lambda\in[0,1)$ with $M\lambda<1$, then $T$ has a unique
	fixed point and the Picard iteration converges to it in $(X,d)$.
\end{corollary}

\begin{remark}
	Corollary~\ref{cor:Banach} complements the Banach type theorem of
	Jleli and Samet \cite{JleliSamet2025}: there, no domination is
	assumed but $T$ is required to be continuous with respect to $d$. Here, no continuity is assumed but the observed distance is
	controlled. Example~\ref{ex:counter} below shows that dropping
	both hypotheses destroys the conclusion, so the two results
	chart the two available routes.
\end{remark}

\begin{proposition}[A priori error bound from observed data]\label{prop:error}
	Under the assumptions of Theorem~\ref{thm:HR}, let $x^\ast$ be the
	unique fixed point of $T$ and let $x_{n+1}=Tx_n$. Then for all
	$n\ge 0$,
	\begin{equation}\label{eq:error-estimate}
		d(x_n,x^\ast)\le \frac{\kappa^{\,n}}{1-\kappa}\,D(x_0,x_1),
	\end{equation}
	with $\kappa$ as in \eqref{eq:kappa}.
\end{proposition}

\begin{proof}
	From the proof of Theorem~\ref{thm:HR}, using $d\le D$,
	\[
	d(x_k,x_{k+1})\le\kappa^{\,k}d(x_0,x_1)\le\kappa^{\,k}D(x_0,x_1).
	\]
	Thus, for $m>n$,
	\[
	d(x_n,x_m)\le\sum_{k=n}^{m-1}d(x_k,x_{k+1})
	\le D(x_0,x_1)\sum_{k=n}^{m-1}\kappa^{\,k},
	\]
	and letting $m\to\infty$ yields \eqref{eq:error-estimate}.
\end{proof}

\begin{remark}
	The bound \eqref{eq:error-estimate} is stated in terms of the observed initial displacement $D(x_0,x_1)$, which is the
	quantity available in applications. The sharper bound with $d(x_0,x_1)$ holds as well but is in general not observable.
\end{remark}

\begin{example}\label{ex:HR-metric}
	If $P\equiv 0$, then $D=d$, condition $(\mathrm{W})$ holds trivially
	with $a_0=b_0=c_0=h_0=e_0=0$ (and condition \eqref{eq:controlled}
	holds with $M=1$), and Theorem~\ref{thm:HR} reduces to the classical
	Hardy-Rogers theorem in a metric space \cite{HardyRogers1973}.
\end{example}

\begin{example}[A non-symmetric observed distance]\label{ex:HR-nonsym}
	Let $X=\mathbb{R}$, let $d(x,y)=|x-y|$, and fix $\beta\in(0,1)$.
	Define
	\[
	P(x,y)=
	\begin{cases}
		\beta\,|x-y|, & x\le y,\\
		0, & x>y,
	\end{cases}
	\qquad
	D(x,y)=d(x,y)+P(x,y).
	\]
	Then $d=D-P$, while $D$ is not symmetric whenever $x<y$, and
	$D(x,y)\le(1+\beta)\,d(x,y)$, so $D$ is controlled by $d$ with
	$M=1+\beta$. Let $\lambda\in[0,1)$ satisfy $(1+\beta)\lambda<1$ and
	define $T(x)=\lambda x$. Since $\lambda\ge0$ preserves the order of
	points, one checks directly that $D(Tx,Ty)=\lambda\,D(x,y)$ for all
	$x,y$. Hence Corollary~\ref{cor:Banach} applies, and $T$ has the
	unique fixed point $0$.
\end{example}

\subsection{Necessity of the hypotheses - a counterexample}

We now show that if condition $(\mathrm{W})$ (in particular, the domination
condition) and the $d$-continuity of
$T$ are both removed, then even the strongest contractive condition - a Banach
type condition in $D$ with coefficient $\tfrac12$ - does not guarantee a
fixed point.

\begin{example}\label{ex:counter}
	Let $X=[0,1]$ with the usual metric $d(x,y)=|x-y|$, so that $(X,d)$ is
	complete. Define
	\[
	P(x,y):=
	\begin{cases}
		2, & \text{if exactly one of } x,y \text{ equals } 0,\\
		0, & \text{otherwise},
	\end{cases}
	\qquad D:=d+P .
	\]
	Then $P\ge0$ is symmetric, vanishes on the diagonal, and $d=D-P$ is a
	metric, so $(X,D,P)$ is a complete perturbed metric space. Define
	\[
	T(0):=1,\qquad T(x):=\frac{x}{2}\quad(0<x\le1).
	\]
	Then:
	\begin{itemize}
		\item[$(a)$] $T$ has no fixed point: $x/2=x$ forces $x=0$, but
		$T(0)=1\ne0$.
		\item[$(b)$] $T$ satisfies the perturbed Banach condition
		\[
		D(Tx,Ty)\le\tfrac12\,D(x,y),\qquad x,y\in X .
		\]
		Indeed, for $x,y\in(0,1]$ we have $Tx,Ty\in(0,\tfrac12]$, so
		$P(Tx,Ty)=0$ and
		$D(Tx,Ty)=\tfrac12|x-y|\le\tfrac12 D(x,y)$. For $x=0$ and
		$y\in(0,1]$ we have $T0=1$, $Ty=y/2\in(0,\tfrac12]$, so
		$P(T0,Ty)=0$ and
		\[
		D(T0,Ty)=\left|1-\tfrac{y}{2}\right|<1
		\le\tfrac12\,(y+2)=\tfrac12\,D(0,y).
		\]
		The case $y=0$, $x\in(0,1]$ is symmetric, and the diagonal case
		is trivial since $D(Tx,Tx)=P(Tx,Tx)=0$.
	\end{itemize}
	Note where the classical mechanisms fail. The Picard sequence started
	at $x_0=1$ is $x_n=2^{-n}$, which $d$-converges to $0$, but the limit
	is not a fixed point, because $T$ is not $d$-continuous at $0$ and no
	estimate transfers the contraction across the limit. The perturbation
	$P(0,y)=2$ does not vanish as $y\to0$, so
	$\sup_{y\ne0}P(0,y)/d(0,y)=\sup_{y}2/y=\infty$ and $D$ is not
	controlled by $d$ for any $M$.

	Condition $(\mathrm{W})$ fails as well, for every admissible choice of
	constants. Here $a=\tfrac12$, $b=c=h=e=0$, so \eqref{eq:W} reads
	$\tfrac12 P(x,y)\le a_0\,d(x,y)+b_0\,d(x,Tx)+c_0\,d(y,Ty)
	+h_0\,d(x,Ty)+e_0\,d(y,Tx)$, and \eqref{eq:HR-coeff} forces
	$a_0+b_0+c_0+h_0+e_0<\tfrac12$. Taking $x=0$ and $y\in(0,1]$, the
	left-hand side equals $1$, while, using $T0=1$ and $Ty=y/2$, the
	right-hand side equals
	$a_0y+b_0+c_0\tfrac{y}{2}+h_0\tfrac{y}{2}+e_0(1-y)
	\le a_0+b_0+c_0+h_0+e_0<\tfrac12$, a contradiction. This is of
	course consistent with Theorem~\ref{thm:HR}, since $T$ has no fixed
	point.
\end{example}

\begin{remark}\label{rem:NP} Example~\ref{ex:counter} is compatible with the Banach type theorem of	Jleli and Samet \cite{JleliSamet2025}, since the map $T$ above is not $d$-continuous.
\end{remark}

We now locate precisely the gap Example~\ref{ex:counter} exposes in Theorem~2.2 of \cite{NutuPacurar2025}, which asserts a fixed point for every $\varphi$-perturbed mapping $D(Tx,Ty)\le\varphi(D(x,y))$ for a comparison function $\varphi$) on a complete perturbed metric space, with no further hypothesis. Taking $\varphi(t)=t/2$, Example~\ref{ex:counter} is precisely of this form. The proof of \cite[Thm.~2.2]{NutuPacurar2025} shows that the Picard sequence $(x_n)$ is Cauchy in $(X,d)$, hence	$d$-convergent to some $x^\ast$, and then estimates $D(Tx_n,Tx^\ast)\le\varphi\bigl(D(x_n,x^\ast)\bigr)$ and let $n\to\infty$ to conclude $D(x^\ast,Tx^\ast)\le\varphi(0)=0$. This last step requires $D(x_n,x^\ast)\to0$, but only $d(x_n,x^\ast)\to0$
has been established, and $D(x_n,x^\ast)=d(x_n,x^\ast)+P(x_n,x^\ast)$	need not tend to $0$ if the perturbation term $P(x_n,x^\ast)$ does not	vanish along the orbit. Example~\ref{ex:counter} realizes exactly this failure: $d(x_n,0)=2^{-n}\to0$ while $P(x_n,0)=2$ for every $n\ge1$, so	$D(x_n,0)\to2\ne0$ and $\varphi(D(x_n,0))\to\varphi(2)=1\ne0$, and	indeed $T$ has no fixed point. \\

The missing hypothesis is therefore not, as one might first guess, global continuity of $T$ with respect to $D$ or with respect to $d$. The mapping of \cite[Ex.~2.4]{NutuPacurar2025}, which the authors use to illustrate Theorem~2.2 on a genuinely discontinuous $T$, is	discontinuous only away from its fixed point and is continuous at the fixed point (constant on a whole neighbourhood of it), which is exactly what the limiting step above needs. The sharp	missing hypothesis is local. Either $d$-continuity of $T$ at the limit point $x^\ast$ together with $P(x_n,x^\ast)\to0$, or, as in the present paper, the domination condition	\eqref{eq:controlled}, which forces $D(x_n,x^\ast)\le M\,d(x_n,x^\ast)\to 0$	automatically and so repairs the limiting step without any continuity assumption on $T$. More generally, condition $(\mathrm{W})$ of the present paper sidesteps the $D$-limit entirely. Via	Lemma~\ref{lem:conversion} the whole argument is transferred to the exact metric $d$, and $D(x_n,x^\ast)\to0$ is never needed. This gap is specific to the $\varphi$-perturbed Theorem~2.2 and does not affect the perturbed Kannan theorem (\cite[Thm.~3.3]{NutuPacurar2025}). There the analogous estimate is $D(Tx_n,Tx^\ast)\le\lambda\bigl[D(x_{n-1},x_n)+D(x^\ast,Tx^\ast)\bigr]$, in which the term $D(x^\ast,Tx^\ast)$ is a constant (not evaluated
along the orbit) and $D(x_{n-1},x_n)\to0$ is obtained directly from the Kannan contractive inequality in $D$, without ever requiring $D(x_n,x^\ast)\to0$. Example~\ref{ex:counter} is not a $\varphi$-perturbed Kannan mapping and does not bear on that theorem.

\subsection{Ulam-Hyers stability and data dependence}

Throughout this subsection, $(X,D,P)$ is a complete perturbed metric space
and $T:X\to X$ satisfies the Hardy-Rogers condition \eqref{eq:HR-D} together
with condition $(\mathrm{W})$ and the coefficient condition
\eqref{eq:HR-coeff}. By Theorem~\ref{thm:HR}, $T$ has a unique fixed point
$x^\ast$. As before, $A=a+a_0$, $B=b+b_0$, $C=c+c_0$, $H=h+h_0$, $E=e+e_0$,
and $\sigma=A+B+C+H+E<1$. In the dominated case of
Corollary~\ref{cor:dominated} these reduce to $A=Ma$, $B=Mb$, $C=Mc$,
$H=Mh$, $E=Me$. \\

A point $u\in X$ is an observed {$\varepsilon$-solution} of the fixed
point equation $x=Tx$ if
\begin{equation}\label{eq:eps-sol}
	D(u,Tu)\le\varepsilon .
\end{equation}
This is the natural notion in the perturbed setting. The residual is measured
in the distance actually available to the observer. Since $d\le D$, every
observed $\varepsilon$-solution is an $\varepsilon$-solution in the exact
metric as well (i.e. $d(u,Tu) \leq \varepsilon$).

\begin{theorem}[Ulam-Hyers stability]\label{thm:UH}
	For every $\varepsilon>0$ and every $u\in X$ with
	$D(u,Tu)\le\varepsilon$,
	\begin{equation}\label{eq:UH}
		d(u,x^\ast)\ \le\ \frac{1+B+E}{\,1-(A+H+E)\,}\;\varepsilon.
	\end{equation}
	In particular, the fixed point equation of $T$ is Ulam-Hyers stable,
	with a constant explicit in the data
	$(a,b,c,h,e,a_0,b_0,c_0,h_0,e_0)$.
\end{theorem}

\begin{proof}
	Put $r:=d(u,x^\ast)$ and $\delta:=d(u,Tu)\le D(u,Tu)\le\varepsilon$.
	By the triangle inequality for $d$,
	\begin{equation}\label{eq:UH1}
		r\le \delta+d(Tu,x^\ast)=\delta+d(Tu,Tx^\ast).
	\end{equation}
	Apply \eqref{eq:HR-d} with $x=u$, $y=x^\ast$, using
	$d(x^\ast,Tx^\ast)=0$, $d(u,Tx^\ast)=r$ and
	$d(x^\ast,Tu)\le d(x^\ast,u)+d(u,Tu)=r+\delta$:
	\[
	d(Tu,Tx^\ast)\le A\,r+B\,\delta+H\,r+E\,(r+\delta)
	=(A+H+E)\,r+(B+E)\,\delta .
	\]
	Substituting into \eqref{eq:UH1},
	\[
	r\le (A+H+E)\,r+(1+B+E)\,\delta,
	\]
	and since $A+H+E\le\sigma<1$, rearranging gives
	$r\le\frac{1+B+E}{1-(A+H+E)}\,\delta
	\le\frac{1+B+E}{1-(A+H+E)}\,\varepsilon$.
\end{proof}

\begin{remark}[Non-sharpness]\label{rem:UH-notsharp}
	The constant in \eqref{eq:UH} is a valid upper bound but is not
	claimed to be optimal. For instance, in the dominated case of
	Corollary~\ref{cor:dominated} with a pure Banach-type contraction
	($b=c=h=e=0$, so $B=C=H=E=0$, $A=Ma$), a direct computation on the
	orbit $u\mapsto D(u,Tu)/D(u,x^\ast)$ ratio gives the sharper worst-case
	value $1/(M-A)$, strictly smaller than the bound $1/(1-A)$ produced by
	\eqref{eq:UH} whenever $M>1$. The gap is the factor $(M-A)/(1-A)$.
	Determining the optimal Ulam-Hyers modulus in perturbed metric spaces,
	analogous to the sharp modulus for contractions in $b$-metric
	setting determined by the first and third author in \cite{BajPet2026}, is left open.
\end{remark}

\begin{corollary}[Well-posedness]\label{cor:wp}
	If $(u_k)\subset X$ satisfies $D(u_k,Tu_k)\to0$, then $u_k\to x^\ast$
	in $(X,d)$.
\end{corollary}

\begin{proof}
	Apply \eqref{eq:UH} with $\varepsilon=\varepsilon_k:=D(u_k,Tu_k)\to0$.
\end{proof}

\begin{theorem}[Data dependence]\label{thm:dd}
	Let $\widetilde T:X\to X$ be any mapping possessing at least one fixed
	point $\widetilde x$, and suppose the observed discrepancy between the
	operators is uniformly bounded:
	\[
	\sup_{x\in X} D\left(Tx,\widetilde Tx\right)\le\eta .
	\]
	Then
	\begin{equation}\label{eq:dd}
		d(x^\ast,\widetilde x)\ \le\
		\frac{1+C+H}{\,1-(A+H+E)\,}\;\eta .
	\end{equation}
\end{theorem}

\begin{proof}
	Put $r:=d(x^\ast,\widetilde x)$ and note first that
	\begin{equation}\label{eq:dd0}
		d(\widetilde x,T\widetilde x)
		=d\left(\widetilde T\widetilde x,\;T\widetilde x\right)
		\le D\left(T\widetilde x,\widetilde T\widetilde x\right)\le\eta ,
	\end{equation}
	where we used $d\le D$ and the symmetry of $d$. By the triangle
	inequality,
	\begin{equation}\label{eq:dd1}
		r\le d\left(x^\ast,T\widetilde x\right)
		+d\left(T\widetilde x,\widetilde x\right)
		\le d\left(Tx^\ast,T\widetilde x\right)+\eta .
	\end{equation}
	Apply \eqref{eq:HR-d} with $x=x^\ast$, $y=\widetilde x$, using
	$d(x^\ast,Tx^\ast)=0$, \eqref{eq:dd0}, and
	$d(x^\ast,T\widetilde x)\le r+\eta$,
	$d(\widetilde x,Tx^\ast)=d(\widetilde x,x^\ast)=r$:
	\[
	d(Tx^\ast,T\widetilde x)
	\le A\,r+C\,\eta+H\,(r+\eta)+E\,r
	=(A+H+E)\,r+(C+H)\,\eta .
	\]
	Substituting into \eqref{eq:dd1} and rearranging (again
	$A+H+E<1$) yields \eqref{eq:dd}.
\end{proof}

\begin{remark}
	Both \eqref{eq:UH} and \eqref{eq:dd} take their input - the residual
	$\varepsilon$ and the operator discrepancy $\eta$ - in the
	observed distance $D$, while producing conclusions in the exact
	metric $d$. This asymmetric formulation, which appears to be new, is
	the practically relevant one. Residuals and discrepancies are computed
	from measured data, whereas the guarantee one seeks concerns the true
	state. To our knowledge these are the first stability and data
	dependence results in perturbed metric spaces.
\end{remark}

\subsection{A Jungck type theorem for weakly compatible pairs}

Recall that a pair $S,T:X\to X$ is weakly compatible (Jungck \cite{Jungck1996}) if $S$ and $T$ commute at their coincidence points, that is, $STz=TSz$ whenever $Sz=Tz$. Every commuting pair is weakly compatible. The converse fails in general.

\begin{theorem}[Jungck type theorem]\label{thm:Jungck}
	Let $(X,D,P)$ be a complete perturbed metric space with exact metric
	$d=D-P$. Let $S,T:X\to X$ satisfy:
	\begin{itemize}
		\item[$(J_1)$] $S$ and $T$ are weakly compatible;
		\item[$(J_2)$] $T(X)\subseteq S(X)$;
		\item[$(J_3)$] $S(X)$ is closed in $(X,d)$;
		\item[$(J_4)$] there exists $\alpha\in[0,1)$ such that
		\[
		D(Tx,Ty)\le\alpha\,D(Sx,Sy),\qquad x,y\in X;
		\]
		\item[$(J_5)$] there exists $\alpha_0\ge0$ such that
		\[
		\alpha\,P(Sx,Sy)\le\alpha_0\,d(Sx,Sy),\qquad x,y\in X,
		\]
		(the Jungck form of condition $(\mathrm{W})$), and
		$\alpha+\alpha_0<1$.
	\end{itemize}
	Then $S$ and $T$ have a unique common fixed point in $X$.
\end{theorem}

\begin{proof}
	Set $q:=\alpha+\alpha_0\in[0,1)$. Conditions $(J_4)$ and $(J_5)$ together imply
	\begin{multline}\label{eq:J-d}
		d(Tx,Ty)\le D(Tx,Ty)\le\alpha\,d(Sx,Sy)+\alpha\,P(Sx,Sy)\\
		\le q\,d(Sx,Sy),\qquad x,y\in X .
	\end{multline}
	
	Fix $x_0\in X$. By $(J_2)$, choose inductively $(x_n)$ with
	$Sx_{n+1}=Tx_n$, and put $y_n:=Tx_n=Sx_{n+1}$. Then, by
	\eqref{eq:J-d}, for $n\ge1$,
	\[
	d(y_n,y_{n+1})=d(Tx_n,Tx_{n+1})\le q\,d(Sx_n,Sx_{n+1})
	=q\,d(y_{n-1},y_n),
	\]
	so $d(y_n,y_{n+1})\le q^{\,n}d(y_0,y_1)$, and, as before, $(y_n)$ is a
	Cauchy sequence in $(X,d)$. Let $y_n\to u$. Since $y_n=Sx_{n+1}\in S(X)$ and
	$S(X)$ is closed with respect to metric $d$, $u\in S(X)$, so there is some $z\in X$ with $Sz=u$.
	
	For $n\ge1$, using \eqref{eq:J-d},
	\begin{align*}
	d(Tz,u)&\le d(Tz,Tx_n)+d(y_n,u)
	\le q\,d(Sz,Sx_n)+d(y_n,u)\\
	&=q\,d(u,y_{n-1})+d(y_n,u)\ \to\ 0 ,
	\end{align*}
	hence $Tz=u=Sz$. Hence, the point $z$ is a coincidence point.
	
	Applying $(J_1)$ at $z$, we get
	\[
	Su=S(Tz)=T(Sz)=Tu .
	\]
	Then, by \eqref{eq:J-d},
	\[
	d(u,Tu)=d(Tz,Tu)\le q\,d(Sz,Su)=q\,d(u,Su)=q\,d(u,Tu),
	\]
	and $q<1$ forces $Tu=u$. Since $Su=Tu$, also $Su=u$. Thus $u$ is a
	common fixed point. \\
	
	If $v$ is another common fixed point, then
	\[
	d(u,v)=d(Tu,Tv)\le q\,d(Su,Sv)=q\,d(u,v),
	\]
	so $u=v$.
\end{proof}

\begin{proposition}[Error estimate for the Jungck scheme]
	\label{prop:Jungck-error}
	Under the assumptions of Theorem~\ref{thm:Jungck}, with
	$y_n=Tx_n=Sx_{n+1}$ and $u$ the common fixed point,
	\[
	d(y_n,u)\ \le\ \frac{q^{\,n}}{1-q}\;D(y_0,y_1),\qquad n\ge0 .
	\]
\end{proposition}

\begin{proof}
	For $m>n$, $d(y_n,y_m)\le\sum_{k=n}^{m-1}q^{\,k}d(y_0,y_1)$. Let
	$m\to\infty$ and use $d(y_0,y_1)\le D(y_0,y_1)$.
\end{proof}

\begin{remark}\label{rem:Jungck-dominated}
	If $D$ is controlled by $d$ with constant $M\ge1$, then
	$P\le(M-1)d$, so $(J_5)$ holds with $\alpha_0=(M-1)\alpha$ and the
	condition $\alpha+\alpha_0<1$ becomes $\alpha M<1$. The dominated Jungck type theorem is therefore a special case of
	Theorem~\ref{thm:Jungck}. As in Remark~\ref{rem:W-weaker}, condition
	$(J_5)$ is strictly weaker than domination: it constrains the
	perturbation only on pairs of $S$-images. If $P\equiv 0$ and
	$S=\mathrm{id}_X$, Theorem~\ref{thm:Jungck} reduces
	to the Banach contraction principle. If $P\equiv0$, it reduces to the
	classical Jungck theorem \cite{Jungck1976} in the weakly compatible
	form \cite{Jungck1996}.
\end{remark}

\begin{example}
	Let $X=\mathbb{R}$ and define $d$, $P$, $D$ exactly as in
	Example~\ref{ex:HR-nonsym}. Let
	\[
	S(x)=x,\qquad T(x)=\lambda x,
	\]
	where $\lambda\in[0,1)$ and $(1+\beta)\lambda<1$. Then $S,T$ commute
	(hence are weakly compatible), $T(X)\subseteq S(X)=X$, and, as in
	Example~\ref{ex:HR-nonsym}, $D(Tx,Ty)=\lambda\,D(Sx,Sy)$. Since
	$P\le\beta\,d$, condition $(J_5)$ holds with
	$\alpha_0=\lambda\beta$, and
	$\alpha+\alpha_0=\lambda(1+\beta)<1$. Hence
	Theorem~\ref{thm:Jungck} applies, even though $D$ is not symmetric and the unique common fixed point is $0$.
\end{example}

Several directions remain open. First, condition $(\mathrm{W})$ already localizes the control of $P$ to the coefficient-weighted pairs. It would be
interesting to determine whether purely asymptotic conditions - for
instance, the vanishing condition $P(x_n,x)\to0$ along $d$-convergent
sequences, which repairs the limiting step in the Banach case
(see below \ Remark~\ref{rem:NP}) - suffice for the full Hardy-Rogers class,
where the cross terms lack a triangle inequality in $D$. Second, multivalued
versions of the present results, with the discrepancy of
Theorem~\ref{thm:dd} measured by a Pompeiu-Hausdorff functional built from
$D$, appear tractable. Third, the optimality of the constants in
\eqref{eq:UH} and \eqref{eq:dd} - in the spirit of the optimal
Ulam-Hyers moduli recently studied in the $b$-metric setting - is a
natural quantitative question, especially in view of
Proposition~\ref{prop:structure}$(b)$, which places controlled observed
distances inside the quasi-$b$-metric family with coefficient $M$.

\bigskip

\noindent
Du\v{s}an Bajovi\'c \\
University of Banja Luka, \\
Faculty of Electrical Engineering, \\
Bosnia and Herzegovina. \\
e-mail: dusan.bajovic@etf.unibl.org

\bigskip

\noindent
Zoran Mitrovi\'c \\
University of Banja Luka, \\
Faculty of Electrical Engineering, \\
Bosnia and Herzegovina. \\
e-mail: zoran.mitrovic@etf.unibl.org

\bigskip

\noindent
Boris Petkovi\'c \\
University of Banja Luka, \\
Faculty of Natural Sciences and Mathematics, \\
Bosnia and Herzegovina. \\
e-mail: boris.petkovic@pmf.unibl.org

\end{document}